\documentclass[12pt]{article}
\usepackage[utf8]{inputenc}
\usepackage{mathrsfs,amsthm,graphicx,color,verbatim,bbm,amsmath,amsfonts,amssymb,newclude,nicefrac,amsfonts,graphicx,xcolor,enumerate,hyperref,cleveref,tikz,pdfpages,bm,authblk}
\usepackage[ margin=1in]{geometry}

\title{Compositional Approximation Can Strictly Outperform Superpositional Approximation}
\author{\footnotesize \textbf{Dennis Elbr\"achter} and \textbf{ Philipp Petersen}}
\affil{\scriptsize Faculty of Mathematics, University of Vienna, Austria}

\date{}

\newtheorem{theorem}{Theorem}[section]
\newtheorem{corollary}[theorem]{Corollary}
\newtheorem{lemma}[theorem]{Lemma}
\newtheorem{proposition}[theorem]{Proposition}

\newtheorem{definition}[theorem]{Definition}
\newtheorem{remark}[theorem]{Remark}

\newcommand{\N}{\mathbb{N}}
\newcommand{\R}{\mathbb{R}}
\newcommand{\cR}{\mathcal{R}}
\newcommand{\cH}{\mathcal{H}}

\newcommand{\eps}{\varepsilon}
\renewcommand{\d}{\mathrm{d}}
\renewcommand{\phi}{\varphi}
\newcommand{\C}{\mathcal{C}}

\newcommand{\D}{\mathcal{D}}
\newcommand{\M}{\mathcal{M}}

\newcommand{\ba}{\bm{a}}
\newcommand{\bb}{\bm{b}}

\DeclareMathOperator*{\argmin}{arg\,min}
\DeclareMathOperator*{\argmax}{arg\,max}
\DeclareMathOperator*{\infp}{\mathrm{inf\vphantom{p}}}

\newcommand{\<}{\langle}
\renewcommand{\>}{\rangle}
\newcommand{\ip}[1]{\left\langle#1\right\rangle}

\newcommand{\cN}{\mathcal{N}}
\newcommand{\cS}{\mathcal{S}}
\newcommand{\cC}{\mathcal{C}}
\newcommand{\cP}{\mathcal{P}}
\newcommand{\cO}{\mathcal{O}}
\newcommand{\cD}{\mathcal{D}}
\newcommand{\cM}{\mathcal{M}}
\newcommand{\cL}{\mathcal{L}}
\newcommand{\cB}{\mathcal{B}}

\newcommand{\fS}{\mathfrak{S}}

\newcommand{\sgamma}{\gamma^{\mathrm{super}}}

\newcommand{\ogamma}{\gamma^{\mathrm{orth}}}
\newcommand{\Rgamma}{\gamma^{\mathrm{Riesz}}}

\begin{document}
\maketitle

\begin{abstract}
    Many classically studied function classes are known to be approximated optimally by superpositional methods, i.e.\@ with approximants constructed as the linear combination of elements in some dictionary.
    Here optimality means that the uniform approximation error viewed as a function of the number of parameters used has polynomial decay of the highest order achievable by any parametrized method whose parameters can be encoded as a bit string of length proportional, up to logarithmic factors, to the number of parameters.  
    While compositional methods like neural networks are structurally different, their approximation rates can be made comparable by imposing constraints that ensure such a proportional bit string encoding.
    In this work we study function classes exhibiting structural properties that limit superpositional approximation rates to be strictly lower than compositional approximation rates.
    In particular, we construct explicit examples for which there is an arbitrarily large gap.   
    
\end{abstract}

\newcommand{\bettersim}{{\,\scriptstyle \sim\,\,}}
\section{Introduction}

While there exists undeniable empirical evidence for the power of neural networks, their approximation-theoretic justification is still rather unclear. Even though there has been an abundance of positive expressivity results for neural networks\footnote{While there are too many to give an anywhere near complete list, we refer to the survey \cite{DeVore_Hanin_Petrova_2021} as well as \cite{yarotsky:2016ReLU} and the many works based on it.}, they are generally based on emulating established approaches, and therefore achieve rates which are as good as existing methods, but not better. 
While the breadth of results suggests a certain universal applicability, this is contrasted by the drawback of the extreme instability of their parametrization \cite{Top_prop,NEURIPS2019_degen}. 
As such, it remains rather difficult to argue from an approximation theory perspective why neural networks should be preferable unless they can make up for their poor parametrization by exhibiting strictly better rates than stably parametrized methods like $M$-term approximation using a well-behaved dictionary, e.g.\@ a frame or even an orthonormal basis. This is somewhat complicated by the observation that for many classically studied function classes, it has been shown that there are dictionaries which yield optimal $M$-term approximation rates \cite{DDDV}.

Our work is motivated by the idea that in order to better understand neural networks, and more generally compositional approximation, it is essential to study function classes for which they provide superior rates of approximation. 
In this paper we focus on developing a technique to obtain upper bounds on the best $M$-term approximation rate achievable by any dictionary $\cD=(\phi_i)_{i\in\N}$ for a given function class $\cC$ in some Hilbert space $\cH$. The central result that comes from this is recorded in \Cref{thm:main}.  
For now, we are not able to bound the supremum over the approximation rates of all possible dictionaries, i.e.\@ countable ordered subsets of $\cH$, but only those which satisfy a lower Riesz bound. We believe this limitation to be technical rather than fundamental. 

We take some inspiration from the study of smoothness spaces, where it has been shown that for balls of certain types of Besov spaces nonlinear approximation rates can be strictly larger than the linear rate \cite{DeVore_1998}. 
From an asymptotic geometric perspective this can be interpreted as follows. We have a function class that, for any given error $\eps>0$, can be $\eps$-approximated by a linear subspace of dimension $d_\eps\bettersim\eps^{-\frac{1}{\gamma_{\mathrm{linear}}}}$. Moreover, the function class can also be $\eps$-approximated nonlinearly via linear combinations of $M_\eps\bettersim\eps^{-\frac{1}{\gamma_{\mathrm{nonlinear}}}}$ many elements from some dictionary under a polynomial depth search constraint, i.e.\@ the functions in the $M_\eps$-term combinations are chosen from the first $\pi(M_\eps)$ elements of the dictionary, where $\pi$ is some polynomial. 
Geometrically speaking, the set of $M_\eps$-term combinations constitutes a union of $\binom{\pi(M_\eps)}{M_\eps}$ many $M_\eps$-dimensional linear subspaces. We can now compare the ``size'' of the $\eps$-approximating linear subspace and the $\eps$-approximating union of lower-dimensional linear subspaces, where we take our notion of size to be the order of the $\eps$-covering numbers of the respective sets. 
For the $d_\eps$-dimensional linear spaces the covering numbers behave like $\bettersim\! \eps^{-d_\eps}$, whereas unions of $M_\eps$-dimensional linear subspaces have covering numbers $\lesssim\! \eps^{-\deg(\pi)M_\eps\log(M_\eps)}$. Under the assumption of $\gamma_{\mathrm{nonlinear}}>\gamma_{\mathrm{linear}}$ we thus have 
\begin{align*}
    \eps^{-\deg(\pi)M_\eps\log(M_\eps)}\in o(\eps^{-d_\eps}),
\end{align*}
which means that the nonlinear sets that are sufficient to $\eps$-approximate the function class are vanishingly small relative to the size of the necessary linear space as $\eps\to 0$.

The above observation can be viewed as an underlying structural property of the function class. For a nonlinear approximation rate of a function class to be strictly larger than the best linear rate, 
finite-dimensional approximations of this function class need to be found in an asymptotically vanishingly small subset\footnote{Note that in the argument above the $M_\eps$-term approximations do not inherently need to be contained in the corresponding $d_\eps$-dimensional linear space. We could, however, just 
consider their orthogonal projection into the linear space, which would only introduce an additional $\cO(\eps)$ error and thus not significantly change the argument.} with a specific structure, in this case being a union of a suitably limited number of lower-dimensional subspaces. In the above example, this 
structural property allows for better nonlinear than linear approximation rates. Our goal is now to utilize this asymptotic geometric perspective to determine structural properties of a function class which can be exploited by compositional approximation but not by superpositional approximation. 
It is important to note that our notion of compositional approximation has constraints which, similar to polynomial depth search, limit the sets of elements that can be approximated using $M$ parameters to have covering numbers whose logarithm is of order $M$ up to logarithmic terms, i.e.\@ they cannot be significantly larger than in the superpositional setting. The crucial point is that we want to determine types of function classes where compositional approximation is more adapted to take advantage of their underlying asymptotic geometric structure.

While we will introduce a concrete neural network model in order to give explicit examples, it should be noted that we are not particularly interested in any specific variant of neural networks, but more generally in the structural property that approximants are assembled as a composition of simple building blocks instead of a superposition of simple building blocks. Section \ref{sec:framework} introduces our mathematical framework, Section \ref{sec:supLims} develops a way to bound superpositional approximation rates\footnote{For the main result a somewhat restricted notion of superpositional rate is required.} of function classes containing certain problematic sets, and Section \ref{sec:constrOfSpoons} constructs concrete examples of such function classes. A number of technical lemmas and proofs are deferred to the appendix. 

\section{Framework}\label{sec:framework}
\subsection{General Notation and Terminology}
For $r\in\R_+$, we write $[r]=\{1,\dots,\lfloor r \rfloor\}$. We denote by $\cP_+$ the set of all univariate polynomials with non-negative coefficients. We will be employing Landau notation to indicate the asymptotic behavior of functions, namely we define 
\begin{align*}
    \cO_{M\to\infty}(f(M))&:=\{g\colon \lim_{M\to\infty}\tfrac{g(M)}{f(M)}<\infty\}\subseteq\{g\colon\N\to\R\},\\
    \cO_{\eps\to 0}(f(\eps))&:=\{g\colon\lim_{\eps\to 0}\tfrac{g(\eps)}{f(\eps)}<\infty\}\subseteq\{g\colon(0,\infty)\to\R\},\\
    o_{M\to\infty}(f(M))&:=\{g\colon \lim_{M\to\infty}\tfrac{g(M)}{f(M)}=0\}\subseteq\{g\colon\N\to\R\},\\
    o_{\eps\to 0}(f(\eps))&:=\{g\colon \lim_{\eps\to 0}\tfrac{g(\eps)}{f(\eps)}=0\}\subseteq\{g\colon(0,\infty)\to\R\},
\end{align*}
where we may omit the index if the argument and asymptotic direction are clear from context. We will use $\|\cdot\|=\|\cdot\|_{\cH}$ and $\<\cdot,\cdot\>=\<\cdot,\cdot\>_{\cH}$ in order to denote the norm and inner product, respectively, in the Hilbert space $\cH$, where we may omit the subscript $\cH$ for simplicity of presentation. Any norm that is not the norm of the ambient Hilbert space will be explicitly denoted as such. We also use the brackets $\<x_1,\dots,x_n\>$ to denote the linear span of a set of vectors, as usage should be clear from context. We will also use $\|\cdot\|_0$ to denote the number of nonzero entries of a vector, even though it is not actually a norm. For matrices, we will use $\|A\|_\infty:=\max_{i,j}|A_{i,j}|$. We denote the unit ball as $B_1(\cH)$ and the unit sphere as $S_1(\cH)$.

\subsection{Approximation schemes and rates}

We begin by introducing the notion of an approximation scheme, which is essentially just a general way to formalize the notion of approximations with a budget $M$ without needing to specify in advance what kind of approximants are used.
Our results crucially rely on notions of orthogonality, and, as such, we will always be working in a Hilbert space. For the results on superpositional approximation the specific nature of the Hilbert space is largely irrelevant, whereas for compositional approximation to make sense we, of course, need a Hilbert space of functions, e.g.\@ the Lebesgue space of square-integrable functions on some domain in $\R^d$. 

\begin{definition}[Approximation Scheme]
Let $\cH$ be a Hilbert space. An \textbf{approximation scheme} is a sequence $\cS=(S_M)_{M\in\N}$ of sets $S_M\subseteq\cH$. We denote by $\fS(\cH)$ the set of all approximation schemes.
\end{definition}

With this in hand, we define the approximation rate of a scheme to be the polynomial order of decay of the uniform approximation error as a function of the budget $M$. We note that taking the supremum in \eqref{eq:app_rate_def} absorbs logarithmic terms, as $\log(x)\in x^\delta$ for every $\delta>0$. This means that, for $k\in\R$, approximation-error decays of $M^{-\gamma}$ and $M^{-\gamma}\log(M)^k$ are considered to be essentially the same in our setting. 

\begin{definition}[Approximation Rates]\label{def:app_rate}
Let $\cH$ be a Hilbert space, $\cC\subseteq \cH$, and $\cS\in\fS(\cH)$. We define the \textbf{approximation rate} of $\cS$ for $\cC$ as
\begin{align}\label{eq:app_rate_def}
    \gamma(\cC,\cS):=\sup\left\{\gamma>0\colon
    \left(M\to \sup_{f\in\cC}\infp_{h\in S_M}\|f-h\|_{\cH}\right)\in\cO_{M\to\infty}(M^{-\gamma})\right\}.
\end{align}
\end{definition}

Given a dictionary and a polynomial $\pi$, we now define the induced superpositional approximation scheme by setting the set of budget-$M$ approximants to contain all $M$-term combinations, where the index of the used dictionary elements as well as the size of the coefficients are bounded by $\pi(M)$. 

\begin{definition}[Superpositional Schemes]\label{def:super_schemes}
Let $\cH$ be a Hilbert space, $\cD=(\phi_k)_{k\in\N}\subseteq \cH$, and $\pi$ a polynomial. We define the \textbf{superpositional approximation scheme} with \textbf{depth search constraint} $\pi$ as $\cS(\cD,\pi)=(S_M)_{M\in\N}$ with 
\begin{align}\label{eq:dict_SM}
    S_M:=\left\{\sum_{i\in I}c_i \phi_i \colon I\subseteq [\pi(M)], |I|\leq M, c\in\R^I, \|c\|_\infty\leq\pi(M)\right\}.
\end{align}
\end{definition}

We are now able to define precisely what we mean by a superpositional approximation rate, namely an upper bound\footnote{As we are taking a supremum, there does not necessarily need to exist a dictionary that actually achieves an optimal rate.} on the best rate achieved by any dictionary under any polynomial depth search constraint. 

\begin{definition}[Superpositional approximation rate]\label{def:sup_rate}
Let $\cH$ be a Hilbert space and $\cC\subseteq \cH$.
We define the \textbf{superpositional approximation rate} for $\cC$ as
\begin{align*}
    \sgamma(\cC):=
    \sup_{\substack{\cD=(\phi_i)_{i\in\N}\subseteq\cH, \pi\in\cP_+}} \gamma(\cC,\cS(\cD,\pi)).
\end{align*}
\end{definition}

Note that the polynomial depth search constraint is crucial, as we want to be very general with our notion of dictionaries and therefore simply define them to be arbitrary sequences of elements in the Hilbert space. 
Function classes studied in approximation theory are generally separable, i.e.\@ contain countable dense subsets which would allow arbitrarily good $1$-term approximations and therefore an infinite rate.

\begin{definition}[Optimal rate]
Let $\cH$ be a Hilbert space, let $\cC\subseteq \cH$, and denote the \textbf{covering numbers} of $\cC$ by
\begin{align*}
    N_\eps(\cC):=\inf\{|X|\colon X \subseteq\cH, \cC\subseteq\bigcup_{x\in X}B_\eps(x)\}\in\N\cup\{\infty\},
\end{align*}
where $B_\eps(x):=\{y\in\cH\colon \|x-y\|_{\cH}\leq\eps\}$. We define the \textbf{optimal rate} of $\cC$ as
\begin{align*}
    \gamma^*(\cC):=\sup\left\{\gamma>0\colon
    \left(\eps\to \log N_\eps(\cC)\right)\in\cO_{\eps\to 0}(\eps^{-\frac{1}{\gamma}})\right\}.
\end{align*}
\end{definition}

This provides us with a measure of complexity of a function class. In particular, it can be shown\footnote{See \cite{DeepTheory} for a proof as well as a more detailed discussion of the matter.} that
\begin{align}\label{eq:super_optimal_bound}
    \sgamma(\cC) \leq \gamma^*(\cC).
\end{align}
Essentially, due to the polynomial constraint, any $M$-term approximation can be encoded as a bit string of length $\cO(M\log^q(M))$, where the specific choice of $\pi$ only affects $q$ and the implicit constant, and the quantization error is no larger than the approximation error. This means that the possibility to $\eps$-approximate a set with $M$-term combinations induces a $2\eps$-covering of cardinality $2^{\cO(M\log^q(M))}$. 
Importantly, \eqref{eq:super_optimal_bound} is not just an upper bound, but for many relevant function classes it is attainable.

Note that, if we allowed an exponentially deep search, one could obtain the optimal rate by only using $1$-term approximations from a dictionary which is simply an amalgamation of $\eps_k$-coverings, for some suitable sequence of $\eps_k$ with $\lim_{k\to\infty}\eps_k=0$. Intuitively speaking, the polynomial depth search constraint ensures that the expressivity of a superpositional scheme relies on the combinatorial power of taking different linear combinations of dictionary elements.

Unfortunately, we are currently only able to bound $\sgamma(\cC)$ in a rather pathological case, whereas for our main results we need to put some restrictions on the allowed dictionaries. Thus, we introduce the following.

\begin{definition}[Riesz Dictionaries]\label{def:Riesz_dict}
Let $T\subseteq\N$, $\cD=(\phi_i)_{i\in T}\subseteq\cH$, and $c\in(0,\infty)$.
We call $\cD$ a \textbf{$c$-Riesz dictionary} if, for every $i\in T$, we have $\|\phi_i\|_{\cH}=1$, and for every finitely supported sequence $(a_i)_{i\in T}\subseteq\R$, we have
\begin{align*}
    \|\sum_{i\in T} a_i\phi_i\|_{\cH}^2\geq c \sum_{i\in T} |a_i|^2.
\end{align*}
For $c\in(0,\infty)$, we denote by $\mathrm{Riesz}_c(\cH)$ the set of all $c$-Riesz dictionaries in $\cH$.
\end{definition}

\begin{definition}[Riesz approximation rate]
Let $\cH$ be a Hilbert space and $\cC\subseteq \cH$.
We define the \textbf{Riesz approximation rate} for $\cC$ as
\begin{align*}
    \Rgamma(\cC):=
    \sup_{\substack{c\in(0,\infty),\cD\in \mathrm{Riesz}_c(\cH), \pi\in\cP_+}} \gamma(\cC,\cS(\cD,\pi)).
\end{align*} 
\end{definition}

Next, we introduce some neural network terminology in order to have an example of a compositional approximation scheme. 

\begin{definition}[Neural Networks]
Let $L\in\N$, $N=(N_0,\dots,N_L)\in\N^{L+1}$, and $\rho\colon\R\to\R$. We call an $L$-tuple of matrix-vector pairs 
\begin{align*}
    \Theta=(A_\ell,b_\ell)_{\ell=1}^L\in\prod_{\ell=1}^L(\R^{N_\ell\times N_{\ell-1}}\times\R^{N_\ell})=:\cP_N
\end{align*}
a \textbf{neural network} and define its \textbf{realization} (with respect to $\rho$) as
\begin{align*}
    \cR_\rho(\Theta):= W_L\circ\rho\circ W_{L-1}\circ\dots\circ\rho\circ W_1, 
\end{align*}
where $W_\ell(x):=A_\ell x+b_\ell$ and $\rho$ is applied componentwise. 

We further denote by $\cL(\Theta):=L$ its \textbf{depth}, by $\cM(\Theta):=\sum_{\ell=1}^L(\|A_\ell\|_0+\|b_\ell\|_0)$ its \textbf{connectivity}, and by $\cB(\Theta):=\max_{\ell\\\in[L]}\max\{\|A_\ell\|_\infty,\|b_\ell\|_\infty\}$ its \textbf{weight magnitude}. Lastly, we denote by 
\begin{align*}
    \cN_{d,d'}:=\bigcup_{L\in\N}\bigcup_{\substack{N\in\N^{L+1}\colon\\ N_0=d,N_L=d'}} \cP_N
\end{align*}
the set of all neural networks with $d$-dimensional input and $d'$-dimensional output.
\end{definition}

We can now use this to define a corresponding type of approximation scheme. 
\begin{definition}[Neural Network Schemes]\label{def:NNSchemes}
Let $d,d'\in\N$, $\Omega\subseteq\R^d$, let $\cH\subseteq\{f:\Omega\to\R^{d'}\}$ be a Hilbert space, $\pi\in\cP_+$, and $\rho\colon\R\to\R$. We define the \textbf{neural network approximation scheme} with \textbf{depth constraint} $\pi$ and \textbf{activation} $\rho$ as $\cS_\cN(\rho,\pi)=(S_M)_{M\in\N}$ with 
\begin{align*}
    S_M:=\left\{ \cR_\rho(\Theta) \colon \Theta\in\cN_{d,d'}, \cM(\Theta)\leq M, \cL(\Theta)\leq \pi(\log(M)), \cB(\Theta)\leq\pi(M)\right\}.
\end{align*}
We define, for $\cC\subseteq\cH$, the \textbf{neural network approximation rate} (with activation $\rho$) as
\begin{align*}
    \gamma^{\cN}_\rho(\cC):=
    \sup_{\substack{\pi\in\cP_+}} \gamma(\cC,\cS_\cN(\rho,\pi)).
\end{align*}
\end{definition}

The polylogarithmic constraint on the depth ensures that we can again encode any approximant in $S_M$ as a bit string of length $\cO(M\log^q(M))$, where the specific choice of $\pi$ only affects $q$ and the implicit constant, and the quantization error is no larger than the approximation error. In particular, we have, analogously to the superpositional case, that
\begin{align}\label{eq:NN_optimal_bound}
    \gamma^{\cN}_\rho(\cC) \leq \gamma^*(\cC)
\end{align}
as long as the activation is Lipschitz continuous.

\section{Superpositional Limitations}\label{sec:supLims}

We start by pointing out a fairly simple, but highly pathological, way to obtain a function class with $\sgamma(\cC)=-\infty$ but\footnote{Note that a rate of $-\infty$ means that the set in \eqref{eq:app_rate_def} is empty since the approximation error decays slower than $\!M^{-\gamma}$ for every $\gamma>0$.} $\gamma^{\cN}_\rho(\cC)>0$. Essentially, this boils down to the observation that there are orthonormal systems $(e_k)_{k\in\N}\subseteq\cH$ that can be efficiently approximated by neural networks, i.e.\@ with a connectivity that has logarithmic dependence on the index $k$, and at most polynomial dependence on the approximation error, while obeying the necessary constraints on depth and weight magnitude. In \cite{DeepTheory}, it is shown\footnote{In fact the results in \cite{DeepTheory} are a bit stronger than is needed here, as they consider the case of generator functions that are well-behaved enough for the dependence of the approximation error to also be logarithmic.} that, using ReLU neural networks, this is the case for affine systems under fairly mild assumptions, with spline wavelets as an explicit example. We can now consider the rather degenerate function class $\C:=(\log(k)^{-1}e_k)_{k\in\N}$, and note that $\eps$-approximation of the whole function class is accomplished by providing $\eps$-approximations for all $e_k$ with $k\leq\lceil2^{\frac{1}{\eps}}\rceil$, as the remaining elements are simply $\eps$-approximated by $0$. Due to the logarithmic dependence on the index, this means ReLU neural networks with connectivity having a polynomial dependence on $\tfrac{1}{\eps}$ are sufficient, i.e.\@ we have $\gamma^{\cN}_\rho(\cC)>0$. On the other hand, superpositional approximation requires a dictionary whose size is proportional to the number of elements that need to be approximated in order to achieve a given uniform $\eps$-error, i.e.\@ $\sim\lceil2^{\frac{1}{\eps}}\rceil$, which results in $\sgamma(\cC)=-\infty$. 

A formal result on a slightly extended version of this example can be found in \Cref{sec:lower_end}. The fact that taking linear combinations is not particularly helpful if everything we are trying to approximate is orthogonal to everything else is neither surprising nor, by itself, all that useful, as imposing pairwise orthogonality on all elements of a function class is extremely limiting. If, however, we relax the assumption of orthogonality to almost orthogonality, the situation becomes much more interesting. Specifically, for $\eta\in(0,1)$, two elements $x,y\in\cH$ are called $\eta$-almost orthogonal if
\begin{align*}
    |\<\tfrac{x}{\|x\|_{\cH}},\tfrac{y}{\|y\|_{\cH}}\>|\leq \eta,  
\end{align*}
i.e.\@ we require the (normalized) inner product between the two elements only to be small but not actually $0$. It is a relatively well-known counterintuitive observation\footnote{See Theorem 1 in \cite{KAINEN19937}.} that the $d$-dimensional Euclidean sphere contains large sets of almost orthogonal points. Specifically, it holds for $\eta\in(0,1)$, $d\geq 2$, that
\begin{align*}
    \max\{|X|\colon X\subseteq S^{d-1}, x\neq y\in X \implies |\<x,y\>|\leq\eta\}\geq e^{\frac{1}{2}d\eta^2}.
\end{align*}
This means that for $\eta\sim d^{-\alpha}$ with $\alpha\in(0,\tfrac{1}{2})$, $\eta$ goes to $0$ as $d$ increases, while the potential size of pairwise $\eta$-almost orthogonal sets is exponential in $d^{1-2\alpha}$, which stands in stark contrast to the cardinality of strictly orthogonal sets being at most $d$. Combining this with our geometric asymptotic viewpoint suggests that there may be a fairly rich zoo of function classes that are approximated reasonably well by a sequence of finite-dimensional linear spaces, but that, within these spaces, contain very large\footnote{More specifically, superpolynomially large w.r.t.\@ the finite-dimensional space they are contained in.} sets of almost orthogonal functions. Taking into account the previously discussed difficulty of superpositional approximation in dealing with highly orthogonal sets, it stands to reason that it may also struggle with highly almost orthogonal sets. In the remainder of this section, we will substantiate this intuition and then construct explicit examples in the following section.

To this end, we first establish a technical lemma, showing that, given a large enough set of almost orthogonal elements and any point on the unit sphere $S_1(\cH)$, at least one of them must have a certain distance from that point.
   
\begin{lemma}\label{lem:main_technical}
       Let $\cH$ be a Hilbert space, $\theta > 1$, $T\in\N \cap [4^{\frac{1}{\theta-1}},\infty)$, and $x,z_1,\dots,z_T\in S_1(\cH)$.
       If $|\<z_i,z_j\>|\leq T^{-\theta}$ for all $i,j\in[T]$ with $i\neq j$, then
       \begin{align*}
           \min_{j\in[T]}|\<x,z_j\>| \leq 4 T^{-\frac{1}{2}}.
       \end{align*}
   \end{lemma} 
\begin{proof}
    See \Cref{sec:proofs}.
\end{proof}

As a direct corollary we obtain a bound on the number of almost orthogonal points which can be contained in a spherical cap. 

\begin{corollary}\label{cor:ReachBound}
       Let $\cH$ be a Hilbert space, $\theta > 1$, $n\in\N$, $T\in\N \cap [4^{\frac{1}{\theta-1}},\infty)$, and $x,z_1,\dots,z_n\in S_1(\cH)$.
       If $|\<z_i,z_j\>|\leq T^{-\theta}$, for all $i,j\in[n]$ with $i\neq j$, then
       \begin{align*}
           |\{j\in[n]\colon |\<z_j,x\>|>4T^{-\frac{1}{2}}\}| < T.
       \end{align*}
\end{corollary}

We are now ready to introduce a central object of our theory: a prototypical type of set for which we can obtain an upper bound on the best rate of approximation achievable by any dictionary satisfying a lower Riesz bound (as defined in \Cref{def:Riesz_dict}). This will then allow us to bound the Riesz approximation rate for a given function class if it contains such a prototypical set. 

\begin{definition}\label{def:SPOON}
Let $\cH$ be a Hilbert space and $\mu,\nu\in(0,\infty)$. We call 
$V=(V_m)_{m\in\N}$ with $V_m\subseteq\cH$ a $(\mu,\nu)$-SPOON (\textbf{S}uper\textbf{po}lynomially growing almost \textbf{o}rthogo\textbf{n}al sequence) if there exist $c,\alpha\in(0,\infty)$ and $m_0\in\N$ such that, for every $m\geq m_0$, it holds that
\begin{enumerate}[(i)]
    \item $|V_m|\geq 2^{m^\alpha}$,
    \item $\displaystyle \|v\| = m^{-\mu}\ \forall v\in V_m$, 
    \item $\displaystyle\max_{v,v'\in V_m, v\neq v'}|\<\tfrac{v}{\|v\|},\tfrac{v'}{\|v'\|}\>|\leq cm^{-\nu}$.
\end{enumerate}
With a slight abuse of notation, we will also write $V=\bigcup_{m\in\N}V_m$ for the union over the sequence of sets, as the usage should be clear from context.
\end{definition}

These prototypical sets are structured as a union over a sequence of sets, which can be thought of as levels of increasingly fine details that need to be dealt with as soon as the desired approximation error $\eps$ is small enough such that $m^{-\mu}>\eps$, i.e.\@ $V_m$ can no longer be trivially $\eps$-approximated by $0$. The $V_m$ themselves are simply sets whose cardinality grows superpolynomially in $m$ and whose elements are pairwise $\eta$-almost orthogonal with $\eta$ decaying like an order-$\nu$ polynomial in $m$.
We will now establish the central theorem of this paper and show that such a structure is sufficient to limit the best achievable Riesz approximation rate. 

\begin{theorem}
\label{thm:main}
Let $\cH$ be a Hilbert space,  $\cC\subseteq\cH$, $\mu,\nu\in(0,\infty)$, and $V\subseteq\cH$ a $(\mu,\nu)$-SPOON.
If $V\subseteq\cC$, it holds that
\begin{align}
    \Rgamma(\cC)\leq \tfrac{1}{2}+\tfrac{\mu}{\nu}.
\end{align}
\end{theorem}

\begin{proof}

    Assume, towards a proof by contradiction, that $\Rgamma(\cC)>\tfrac{1}{2}+\tfrac{\mu}{\nu}$. This means that there are $r\in (\tfrac{1}{2}+\tfrac{\mu}{\nu}, \Rgamma(\cC))$, $c>0$, $\cD=(\phi_i)_{i\in\N}\in\mathrm{Riesz}_c(\cH)$, $\tau\in\N$, $C\in\R_+$, and $M_0\in\N$ such that, for all $M\geq M_0$,
    \begin{align*}
        \eps_M(\cC)=\sup_{f\in\cC}\infp_{h\in \<\cD_M\>_M}\|f-h\|_{\cH}\leq C M^{-r},
    \end{align*}
    where $\cD_M=(\phi_i)_{i\in[M^\tau]}$ and
    \begin{align*}
        \<\D_M\>_M=\bigcup_{I\subseteq [M^\tau],|I|\leq M} \<(\phi_i)_{i\in I}\>.
    \end{align*}
We denote by $V=(V_m)_{m\in\N}\subseteq\cC$ the $(\mu,\nu)$-SPOON contained in $\cC$, and observe that, for every $m\in\N$, $M\geq M_0$, we have
\begin{align*}
    \eps_M(V_m)\leq C M^{-r}.
\end{align*}
As this needs to hold for every $M\geq M_0$, we can, in the following, always assume $M$ is as large as is necessary for a given estimate to hold. 
Moreover, we can assume\footnote{Otherwise we could just use finite subsets of the $V_m$ that are large enough to still satisfy condition (i) in \Cref{def:SPOON}, as conditions (ii) and (iii) in \Cref{def:SPOON} remain valid when taking subsets.} w.l.o.g.\@ that $|V_m|<\infty$ for all $m\in\N$.
Let $A_{M}\colon \cH\to \<\cD_M\>_M$ satisfy for all $v\in\cH$ that
\begin{align*}
    A_M(v) \in \argmin_{x\in \<\cD_M\>_M}\|x-v\|,
\end{align*}
and let $R_M\colon \cH\to \cH, v\mapsto v-A_M(v)$. Here the argmin is non-empty because $\<\cD_M\>_M \cap B_1(\cH)$ is compact. The compactness follows immediately since $\cD_M$ is a finite set of vectors. 
It holds, for $v,w\in V_m$, that 
\begin{align*}
    |\<A_M(v),A_M(w)\>|\leq |\<v,w\>| + |\<R_M(v),R_M(w)\>| + |\<R_M(v),w\>| + |\<R_M(w),v\>|.
\end{align*}
While, in general, the last two terms can be of order $M^{-r}m^{-\mu}$, we can, however, argue that for a suitable choice of $m$ and a sufficiently large subset of $V_m$ they obey a bound of order $M^{-r-\frac{1}{2}}m^{-\mu}$. To this end, we consider $\sigma\in(\tfrac{1}{\nu},\tfrac{1}{\mu}(r-\tfrac{1}{2}))$, $m\in(M^\sigma,2M^\sigma)$,
 and define, for $v\in V_m$, the set
\begin{align*}
    P_v:=\{w\in V_m\backslash\{v\}\colon |\<R_M(v),w\>|>4M^{-\frac{1}{2}}\|R_M(v)\|\|w\|\}.
\end{align*}
Applying \Cref{cor:ReachBound} with $T=M$ and $\theta=\tfrac{1}{2}(\sigma\nu+1)$ yields $|P_v|\leq M$, since $c_Sm^{-\nu}\geq T^{-\theta}$ if $M\geq c_S^{\frac{2}{\sigma\nu-1}}$. Next we consider, for $v\in V_m$, the set
\begin{align*}
    Q_v:=\{w\in V_m\backslash\{v\}\colon |\<R_M(w),v\>|>4M^{-\frac{1}{2}}\|R_M(w)\|\|v\|\}.
\end{align*}
While these sets, in general, could be as large as $|V_m|-1$, we note that 
\begin{align*}
    \sum_{v\in V_m}|Q_v|=\sum_{v\in V_m}|P_v|<M|V_m|.
\end{align*}
Consequently, there exists a set $W_m\subseteq V_m$ with $|W_m|\geq (1-\tfrac{1}{M})|V_m|$ such that for every $v\in W_m$, we have $|Q_v|\leq M^2$. Therefore, given $n\in\N$, every subset $X\subseteq W_m$ with $|X|\geq n(M^2 + 2M)$ contains a subset\footnote{This subset can simply be obtained by a greedy-type algorithm, i.e.\@ picking some $v\in X$ and removing $P_v\cup Q_v$. As $v\in Q_{w} \iff w\in P_v$, this can be repeated $n$ times to obtain a set $Y$ with the claimed properties.} $Y$ with $|Y|=n$ such that, for all $v,w\in Y$, it holds that $v\notin P_{w}\cup Q_{w}$ and $w\notin P_{v}\cup Q_{v}$. Next, we note that for $v,w\in V_m$ with $v\notin P_{w}\cup Q_{w}$ and $w\notin P_{v}\cup Q_{v}$, we have
\small
\begin{align}\label{eq:vw_bound}\begin{split}
    |\<\tfrac{A_M(v)}{\|A_M(v)\|},\tfrac{A_M(w)}{\|A_M(w)\|}\>|
    &\leq 4^{\mu+1}M^{2\sigma\mu}\big(|\<v,w\>| + |\<R_M(v),R_M(w)\>| + |\<R_M(v),w\>| + |\<R_M(w),v\>|\big)\\
    &\leq 4^{\mu+1}M^{2\sigma\mu}\big(c_S M^{-\sigma(\nu+2\mu)} + C^2M^{-2r} + 8CM^{-r-\frac{1}{2}-\sigma\mu} \big)\\
    &\leq 4^{\mu+1}(c_S+C^2+8C)M^{-\min\{\sigma\nu,2(r-\sigma\mu), r+\tfrac{1}{2}-\sigma\mu\}},
\end{split}\end{align}
\normalsize
where we used that $\|A_M(v)\|\geq 2^{-(\mu+1)}M^{-\sigma\mu}$ due to $r-\sigma\mu \geq\tfrac{1}{2}$ and $M\geq 4^{\mu+1}C^2$.
We proceed by noting that, due to $\cD_M$ being $c$-Riesz, \Cref{lem:MbGoodEnough_Riesz} ensures that, for every $x\in\<\cD_M\>_M$, there exists an $i\in[M^\tau]$ with
\begin{align}\label{eq:wphi_close}
    |\<\tfrac{x}{\|x\|},\phi_i\>|\geq c^{\frac{1}{2}}M^{-\frac{1}{2}}.
\end{align}
In particular, we have
\begin{align*}
    A_M(W_m) = \bigcup_{i\in[M^\tau]} S_i,
\end{align*}
where
\begin{align*}
    S_i:=\{x\in A_M( W_m)\colon |\<\tfrac{x}{\|x\|},\phi_i\>|\geq c^{\frac{1}{2}}M^{-\frac{1}{2}}\}.
\end{align*}
We observe that, for $M$ sufficiently large, $A_M$ is injective on $V_m$, and thus $|W_m|=|A_M(W_m)|$, since, for every $w,v\in V_m$ with $v\neq w$, we have
\begin{align*}
    \|A_M(v)-A_M(w)\| \geq \|v\|+\|w\| -2|\<w,v\>|- \|v-A_M(v)\|-\|w-A_M(w)\|>0 
\end{align*}
due to $2m^{-\mu}(1-c_sm^{-\nu})<2CM^{-r}$.
Consequently there exists $i\in[M^\tau]$ such that 
\begin{align*}
    |S_i|\geq\tfrac{|W_m|}{M^\tau}\geq \tfrac{(1-\tfrac{1}{M})|V_m|}{M^\tau}\geq (1-\tfrac{1}{M}) M^{-\tau}2^{M^{\sigma\alpha}}.
\end{align*}
Due to \eqref{eq:vw_bound} there exists a subset $Y\subseteq S_i$ with $|Y|\geq \tfrac{|S_i|}{M^2+2M}$ such that, for every $x,y\in Y$ with $x\neq y$, we have
\begin{align*}
    |\<\tfrac{x}{\|x\|},\tfrac{y}{\|y\|}\>|&\leq 4^{\mu+1}(c_S+C^2+8C)M^{-\min\{\sigma\nu,2(r-\sigma\mu), r+\tfrac{1}{2}-\sigma\mu\}}.
\end{align*}
Applying \Cref{cor:ReachBound} with $T=\lceil \tfrac{16}{c}\rceil M$ and $\theta=\tfrac{1}{2}(1+\eta)>1$, where $\eta=\min\{\sigma\nu,2(r-\sigma\mu), r+\tfrac{1}{2}-\sigma\mu\}$, we obtain, for $M\geq ((\tfrac{32}{c})^\theta4^{\mu+1}(c_S+C^2+8C))^{\frac{2}{\eta-1}}$, that  
\begin{align*}
    |Y|\leq \tfrac{32}{c}M.
\end{align*}
Combining the above yields
\begin{align*}
     2^{M^{\sigma\alpha}}\in \cO(M^{\tau+3})
\end{align*}
which is a contradiction, and thus completes the proof.
\end{proof}

\section{Explicit construction of SPOONs}\label{sec:constrOfSpoons}
Having established an abstract condition that limits achievable Riesz approximation rates, we now need to make sure that there are in fact function classes satisfying that condition and that we can obtain better compositional rates for them. 
To this end, we introduce the following construction.
\begin{definition}\label{def:ToySpoon}
    Let 
    \begin{align*}
        f\colon \R&\to [0,1],
                x\mapsto \begin{cases}2x, &x \in[0,\tfrac{1}{2})\\ 2(1-x), &x\in[\tfrac{1}{2},1]\\
                0, &x\in\R\backslash[0,1]\end{cases}. 
    \end{align*}
    For $m\in\N$ and $a\in\{0,1\}^m$, let $f_a\in L^2([0,1])$ be given by
    \begin{align*}
        f_a(x)=\sum_{j=1}^m a_j f(m(x-\tfrac{j-1}{m})).
    \end{align*}
    For $m,k\in\N$ let
    \begin{align*}
        A_{m,k}&:=\{\ba=(a^1,\dots,a^k)\colon a^1,\dots,a^k\in\{0,1\}^m\},\\
        H_{m,k}&:=\{F_{\ba}=f_{a^k}\circ\dots\circ f_{a^1}\colon \ba\in A_{m,k}\}.
    \end{align*}
\end{definition}

Note that taking the composition $g\circ f$ results in a copy of $g$ on the interval $[0,\tfrac{1}{2}]$ and a mirrored copy of $g$ on the interval $[\tfrac{1}{2},1]$, both dilated by a factor of $2$. Consequently, taking the composition $f_{a^1}\circ f_{a^2}$ with $a^1,a^2\in\{0,1\}^m$ results in a function which, on each subinterval $[\tfrac{j-1}{m},\tfrac{j}{m}]$ for which $a^2_j=1$, exhibits a copy and a mirrored copy of $f_{a^1}$. As such, $H_{m,k}$ is naturally contained in the $2^{k-1}m^k$-dimensional span of $h_1,\dots,h_{2^{k-1}m^k}$, where $h_j$ behaves like a suitably dilated hat function on the interval $[\tfrac{j-1}{2^{k-1}m^k}, \tfrac{j}{2^{k-1}m^k}]$ and is $0$ otherwise. However, $H_{m,k}$ contains only very specific functions exhibiting self-similarity induced by the iterated copying and mirroring effect.
The norms and inner products of elements in $H_{m,k}$ can be easily computed from their defining binary vectors as follows.

\begin{lemma}\label{lem:F_nip}
Let $k,m\in\N$ and $\ba,\bb\in A_{m,k}$ with $\forall j\in[k]\colon a^j,b^j\neq0$. Then
    \begin{align*}
        \|F_{\ba}\|_{L^2([0,1])} = \left(\tfrac{1}{3m^k}\prod_{j=1}^k\|a^j\|_0\right)^{\frac{1}{2}}
    \end{align*}
    and
    \begin{align*}
        \frac{\<F_{\ba}, F_{\bb}\>}{\|F_{\ba}\|_{L^2([0,1])}\|F_{\bb}\|_{L^2([0,1])}}
        =\frac{\prod_{j=1}^k\|a^j\wedge b^j\|_0}{\sqrt{\prod_{j=1}^k\|a^j\|_0\|b^j\|_0}}.
    \end{align*}
\end{lemma}
\begin{proof}
    We start by noting that composing $f$ with itself produces two dilated copies of $f$ on the intervals $[0,\tfrac{1}{2}]$ and $[\tfrac{1}{2},1]$, or more specifically in the notation above $f\circ f=f_a$ with $a=(1,1)$. Based on this we obtain
    \begin{align*}
        F_{\ba}=f_{p(\ba)},
    \end{align*}
    with
    \begin{align*}
        p(\ba) = a^1\otimes(a^2,r(a^2))\otimes\dots\otimes(a^k,r(a^k)) \in\{0,1\}^{2^{k-1}m^k},
    \end{align*}
    where $\otimes$ denotes the Kronecker product and $r(x_1,x_2,\dots,x_m):=(x_m,x_{m-1},\dots,x_1)$ reverses the order of the components of a vector. Therefore we have
    \begin{align*}
        \|F_{\ba}\|^2_{L^2([0,1])} = \|p(\ba)\|_0 \int_0^{2^{-k+1}m^{-k}}(2^{k-1}m^kx)^2\d x=\frac{\prod_{j=1}^k\|a^j\|_0}{3m^k}
    \end{align*}
    and, similarly,
    \begin{align*}
        \<F_{\ba}, F_{\bb}\>
        =\frac{\prod_{j=1}^k\|a^j\wedge b^j\|_0}{3m^k}.
    \end{align*}
\end{proof}

In order to proceed, we will need to establish the following combinatorial estimate. It is more naturally formulated for subsets of $[m]$ instead of binary vectors in $\{0,1\}^m$, which can then easily be identified with each other. 

\begin{lemma}\label{lem:Qset}
    Let $\alpha\in(0,\tfrac{1}{2})$, $\beta\in(0,\alpha)$. Then there exists $m_0\in\N$ such that for $m\geq m_0$ there is a set $Q_{m,\alpha,\beta}\subseteq\cP_{m^\alpha}([m]):=\{I\subseteq [m]\colon |I|=\lfloor m^\alpha\rfloor\}$ such that
    \begin{align*}
        \max_{q,q'\in Q_{m,\alpha,\beta}\colon q\neq q'}|q\cap q'|\leq m^\beta
    \end{align*}
    and
    \begin{align*}
        |Q_{m,\alpha,\beta}|\geq  m^{(\frac{1}{2}-\alpha) m^\beta}. 
    \end{align*}
\end{lemma}
\begin{proof}
    For simplicity of notation we write $a:=\lfloor m^\alpha\rfloor$ and $b:=\lceil \tfrac{m^\beta}{2}\rceil$.
    Let $I\in \cP_{m^\alpha}([m])$, $k\in[2a]\cap 2\N$, and note that
    \begin{align}\label{eq:HamBallEst}
        n_k:=|\{J\in\cP_{m^\alpha}([m])\colon |I\vartriangle J|\leq k\}|\leq \sum_{j=0}^{\frac{k}{2}} \binom{a}{j}\binom{m-a}{j}.
    \end{align}
    This is due to the fact that any $J\in\cP_{m^\alpha}([m])$ with\footnote{Since $I$ and $J$ have the same number of elements, their symmetric difference must be a multiple of 2.} $|I\vartriangle J|= 2j$, $j\in[\tfrac{k}{2}]$, can be obtained by removing some choice of $j$ elements from $I$ and adding some choice of $j$ elements from $[m]\backslash I$, resulting in at most $\binom{a}{j}\binom{m-a}{j}$ possible ways to do so. Now let $Q_k$ be a set obtained by starting with the empty set and repeatedly adding an element $J$ of $\cP_{m^\alpha}([m])$ which satisfies $|I\vartriangle J| > k$ for all $I$ that are already in $Q_k$, until this is no longer possible.
    Due to \eqref{eq:HamBallEst}, it needs to hold that
    \begin{align*}
        |Q_k| \geq \frac{|\cP_{m^\alpha}([m])|}{n_k}
    \end{align*}
    We proceed by obtaining an upper bound on $n_{2t}$ for $t=a - b $.  
    Note that, due to $\alpha < \frac{1}{2}$, we have  $a+t+(t+1)^2\in o(m)$, which means that for sufficiently large $m$ the terms $\binom{a}{j}\binom{m-a}{j}$ are increasing in $j$ and thus
    \begin{align*}
       \sum_{j=0}^{t} \binom{a}{j}\binom{m-a}{j}\leq (t+1)\binom{a}{t}\binom{m-a}{t}.
    \end{align*}
    Consequently we get
    \begin{align*}
        \frac{|\cP_{m^\alpha}([m])|}{n_{2t}}
        &\geq \frac{\binom{m}{a}}{(t+1)\binom{a}{t}\binom{m-a}{t}}\\
        &=\frac{1}{(t+1)}\frac{m!}{a!(m-a)!}\frac{t!(a - t)!}{a!}\frac{t!(m-a-t)!}{(m-a)!}\\
        &=\frac{m!}{(m-a)!}\frac{(m-a-t)!}{(m-a)!}\left(\frac{t!}{a!}\right)^2\frac{(a-t)!}{(t+1)}\\
        &\geq (m-a)^a (m-a)^{-t} a^{-2b}\frac{b!}{(t+1)}\\
        &\geq (\frac{m}{2})^{b} m^{-2\alpha b}\frac{b!}{(t+1)}\\
        &\geq m^{(1-2\alpha)b} \frac{b!}{2^{b} (t+1)}\\
        &\geq m^{(\frac{1}{2}-\alpha) m^\beta},
    \end{align*}
    where we used that $m-a\geq \tfrac{m}{2}$ and $\frac{b!}{2^{b} t}\geq 1$ if $m$ is sufficiently large. Thus, all $q,q'\in Q_k\subseteq \cP_{m^\alpha}([m])$ with $q\neq q'$ satisfy 
    \begin{align*}
        |q\cap q'|=\tfrac{1}{2}(|q|+|q'|-|q\vartriangle q'| )< a-t = \lceil\tfrac{m^\beta}{2}\rceil\leq  m^\beta.
    \end{align*}
    Therefore the set $Q_k$ satisfies the desired properties and the claim is proven.    
\end{proof}

Combining this with \Cref{lem:F_nip} enables us to show that $H_{m,k,\alpha,\beta}$ contains large almost orthogonal sets.

\begin{lemma}\label{lem:WSet}
    Let $\alpha\in(0,\tfrac{1}{2})$, $\beta\in(0,\alpha)$, and $k\in\N$. Then there exists $m_0\in\N$ such that for every $m\geq m_0$ there is a set $H_{m,k,\alpha,\beta}\subseteq H_{m,k}$ such that
    \begin{align}\label{eq:pw_ao}
        \max_{h,h'\in H_{m,k,\alpha,\beta}\colon h\neq h'}\left|\left\langle \tfrac{h}{\|h\|_{L^2([0,1])}},\tfrac{h'}{\|h'\|_{L^2([0,1])}}\right\rangle\right|\leq 2m^{k(\beta-\alpha)}
    \end{align}
    and 
    \begin{align}\label{eq:H_cardinality}
        |H_{m,k,\alpha,\beta}| \geq \tfrac{1}{2k}m^{(\frac{1}{2}-\alpha) m^\beta}.
    \end{align}
\end{lemma}
\begin{proof}
    Note that  
    \begin{align*}
        \iota :\{a\in\{0,1\}^m\colon \|a\|_0=\lfloor m^\alpha\rfloor\}&\to \cP_{m^\alpha}([m])\\
        a&\mapsto \{i\in[m]\colon a_i=1\}
    \end{align*}
    is a bijection. Now let $Q_{m,\alpha,\beta}$ be a set satisfying the properties from \Cref{lem:Qset} and define
    \begin{align*}
        A_{m,k,\alpha,\beta}&:= \{\ba=(a^1,\dots,a^k)\colon \iota(a^1),\dots,\iota(a^k)\in Q_{m,\alpha,\beta}\}\subseteq A_{m,k}.
    \end{align*}
    Combining \Cref{lem:F_nip,lem:Qset} yields, for $\ba,\bb\in A_{m,k,\alpha,\beta}$ with $\forall j\in[k]\colon a^j\neq b^j$, that
    \begin{align*}
        \frac{\<F_{\ba}, F_{\bb}\>}{\|F_{\ba}\|_{L^2([0,1])}\|F_{\bb}\|_{L^2([0,1])}}
        =\frac{\prod_{j=1}^k\|a^j\wedge b^j\|_0}{\sqrt{\prod_{j=1}^k\|a^j\|_0\|b^j\|_0}}
        =\frac{\prod_{j=1}^k|\iota(a^j)\cap \iota(b^j)|}{\lfloor m^\alpha\rfloor^k}
        \leq 2m^{k(\beta-\alpha)},
    \end{align*}
    where we used that $\lfloor m^\alpha\rfloor^k\geq \tfrac{1}{2}m^{\alpha k}$ if $m\geq2^{\frac{k+1}{\alpha}}$. We note that, for $\ba\in A_{m,k,\alpha,\beta}$, we have
    \begin{align*}
        |\{\bb\in A_{m,k,\alpha,\beta}\colon\exists j\in[k]\colon a^j = b^j \}|\leq \sum_{j\in[k]}|\{\bb\in A_{m,k,\alpha,\beta}\colon a^j = b^j \}| \leq k |Q_{m,\alpha,\beta}|^{k-1}.
    \end{align*}
    In particular, for every $X\subseteq A_{m,k,\alpha,\beta}$ and $\ba\in X$, we can obtain a set $X'\subseteq X$ with $|X'|\geq |X| -k |Q_{m,\alpha,\beta}|^{k-1}$ and the property that for every $\bb\in X'\backslash\{\ba\}$ and $j\in[k]$, it holds that $a^j\neq b^j$. This is accomplished simply by removing every element from $X$ that does not satisfy this condition. As the condition of not having any matching components is symmetric, we can repeat this procedure without removing any previously chosen $\ba$.
    Consequently, for every $n\in \N$ with $n(1+k|Q_{m,\alpha,\beta}|^{k-1})\leq |A_{m,k,\alpha,\beta}|$, there exists a subset $A^*_{m,k,\alpha,\beta}\subseteq A_{m,k,\alpha,\beta}$ with $|A^*_{m,k,\alpha,\beta}|\geq n$ and
    \begin{align*}
        \max_{\ba,\bb\in A^*_{m,k,\alpha,\beta}\colon \ba\neq \bb}\left|\left\langle \tfrac{F_{\ba}}{\|F_{\ba}\|_{L^2([0,1])}},\tfrac{F_{\bb}}{\|F_{\bb}\|_{L^2([0,1])}}\right\rangle\right|\leq 2m^{k(\beta-\alpha)}.
    \end{align*}
    We observe 
    \begin{align*}
        \tfrac{|A_{m,k,\alpha,\beta}|}{(1+k|Q_{m,\alpha,\beta}|^{k-1})}\geq \tfrac{1}{2k}|Q_{m,\alpha,\beta}| \geq \tfrac{1}{2k}m^{(\frac{1}{2}-\alpha) m^\beta}
    \end{align*}
    and that $\ba\mapsto F_{\ba}$ is injective on $A_{m,k,\alpha,\beta}$ since $0\notin \iota(Q_{m,\alpha,\beta})$. This ensures the existence of a set $H_{m,k,\alpha,\beta}\subseteq H_{m,k}$ satisfying \eqref{eq:pw_ao} and \eqref{eq:H_cardinality}.
\end{proof}

This allows us to define the following family of SPOONs.

\begin{definition}
    Let $\alpha\in(0,\tfrac{1}{2})$, $\beta\in(0,\alpha)$, $\mu\in(0,\infty)$, $k\in\N$, and $m_0\in\N$ large enough such that $\Cref{lem:WSet}$ guarantees the existence of sets $H_{m,k,\alpha,\beta}$ satisfying \eqref{eq:pw_ao} and \eqref{eq:H_cardinality}.
    We define, for $m\geq m_0$,
    \begin{align*}
        V_m(k,\alpha,\beta,\mu):=\{m^{-\mu }\frac{h}{\|h\|_{L^2([0,1])}}\colon h\in H_{m,k,\alpha,\beta}\}.
    \end{align*}
    Moreover, we set $V_m(k,\alpha,\beta,\mu):=\emptyset$ for $m<m_0$. 
\end{definition}

We note that 
\begin{align*}
    V(k,\alpha,\beta,\mu):=(V_m(k,\alpha,\beta,\mu))_{m\in\N}   
\end{align*}
constitutes a $(\mu,k(\alpha-\beta))$-SPOON according to \Cref{def:SPOON}. Condition (ii) follows directly from the definition, and Conditions (i) and (iii) are ensured by \Cref{lem:WSet} upon noting that $2^{m^{\beta}}\in o(\tfrac{1}{2k}m^{(\frac{1}{2}-\alpha) m^\beta})$. 

We can now formulate the central result of this section.

\begin{proposition}\label{lem:toyspoon_rates}
Let $\rho\colon x\mapsto\max\{0,x\}$, $k\in\N$, $\alpha\in(0,\tfrac{1}{2})$, $\beta\in(0,\alpha)$, and $\mu\in(0,\infty)$. Then
\begin{align*}
    \gamma^{\cN}_\rho(V(k, \alpha,\beta,\mu)) \geq \tfrac{\mu}{\alpha}
\end{align*}
and 
\begin{align*}
    \Rgamma(V(k, \alpha,\beta,\mu))\leq \tfrac{1}{2}+\tfrac{\mu}{k(\alpha-\beta)}.    
\end{align*}
\end{proposition}

\begin{proof}
We start by noting that $f$ can be written as
\begin{align*}
    f(x)= 2\rho(x)-4\rho(x-\tfrac{1}{2})+2\rho(x-1).
\end{align*}
    Thus, for every $m\in\N$ and $a\in\{0,1\}^m$, there exists a neural network $\Theta_a$ with $\cL(\Theta_a)=2$, $\cM(\Theta_a)\leq 9\|a\|_0$, $\cB(\Theta_a)\leq \max\{4,m\}$, and
\begin{align*}
    \cR_\rho(\Theta_a)=f_a=\sum_{j=1}^m a_j f(m(\,\cdot\,-\tfrac{j-1}{m})).
\end{align*}
Note that 
\begin{align*}
    \frac{m^{-\mu}}{\|F_{\ba}\|_{L^2([0,1])}} \leq \sqrt{3}m^{\frac{k}{2}(1-\alpha)-\mu}.
\end{align*}
By \Cref{lem:ReLUcomp} there consequently exists a constant $C_k$ such that, for every $\ba\in A_{m,k}$, there is a neural network $\Psi_{\ba}$ with $\cL(\Psi_{\ba})=2k$, $\cM(\Psi_{\ba})\leq C_k m^\alpha$, $\cB(\Psi_{\ba})\leq \max\{4,m,\sqrt{3}m^{\frac{k}{2}(1-\alpha)-\mu}\}$, and
\begin{align*}
    \cR_\rho(\Psi_{\ba})=\frac{m^{-\mu}}{\|F_{\ba}\|_{L^2([0,1])}}F_{\ba}.
\end{align*}
In particular we have, for $m\leq (\tfrac{M}{C_k})^{\frac{1}{\alpha}}$, that
\begin{align*}
    H_{m,k,\alpha,\beta}\subseteq S_M=\left\{ \cR_\rho(\Theta) \colon \Theta\in\cN_{d,d'}, \cM(\Theta)\leq M, \cL(\Theta)\leq \pi(\log(M)), \cB(\Theta)\leq\pi(M)\right\},
\end{align*}
where $\pi$ is a polynomial of degree $\lceil\tfrac{1}{\alpha}\rceil(\tfrac{k}{2}(1-\alpha)-\mu)$. Thus, we get
\begin{align*}
    \sup_{f\in V_m}\infp_{h\in S_M}\|f-h\|_{L^2([0,1])}=0 
\end{align*}
for $m\leq (\tfrac{M}{C_k})^{\frac{1}{\alpha}}$ as well as 
\begin{align*}
\sup_{f\in V_m}\infp_{h\in S_M}\|f-h\|_{L^2([0,1])} \leq \sup_{f\in V_m}\|f\|_{L^2([0,1])}\leq m^{-\mu}\leq (\tfrac{M}{C_k})^{-\frac{\mu}{\alpha}}   
\end{align*}
for $m > (\tfrac{M}{C_k})^{\frac{1}{\alpha}}$ since $0\in S_M$.
Consequently we have
\begin{align*}
    \sup_{f\in \bigcup_{m\in\N}V_m}\infp_{h\in S_M}\|f-h\|_{L^2([0,1])}\in\cO(M^{-\frac{\mu}{\alpha}}),
\end{align*}
which means $\gamma^{\cN}_\rho(V(k, \alpha,\beta,\mu)) \geq \tfrac{\mu}{\alpha}$.
The upper bound on $\ogamma(\cC)$ follows from \Cref{thm:main} and \Cref{lem:WSet}.
\end{proof}
\begin{remark}
    The proof simplifies for neural networks with the ReLU activation function, as they achieve exact representation up to a certain $m$.
    It should, however, be straightforward to adapt the proof for neural networks with any activation function as long as they can efficiently approximate the basic hat function $f$. 
\end{remark}
\begin{remark}
    Note that for the choice of $\mu=\sqrt{k}$, we obtain examples of function classes for which we get an arbitrarily large gap between compositional and Riesz approximation rates. More specifically,
    \begin{align*}
        \lim_{k\to\infty}\gamma^{\cN}_\rho(V(k, \alpha,\beta,\sqrt{k})) =\infty
    \end{align*}
    whereas
    \begin{align*}
    \lim_{k\to\infty}\textbf{}\Rgamma(V(k, \alpha,\beta,\sqrt{k}))\leq \tfrac{1}{2}.    
    \end{align*}
\end{remark}
Another way to look at this is to consider the sets of all functions in $L^2([0,1])$, which can be approximated by ReLU neural networks with at least a given rate $\gamma$, i.e.\@
\begin{align*}
    \cC_{\rho,\gamma}:=\bigcup_{\substack{\cC\subseteq L^2([0,1])\\\gamma^{\cN}_\rho(\cC)\geq \gamma }} \cC.
\end{align*}
Due to \Cref{lem:toyspoon_rates} we know that
\begin{align*}
    V(k, \alpha,\beta,\mu) \subseteq \cC_{\rho,\gamma}
\end{align*}
if $\tfrac{\mu}{\alpha}\geq\gamma$. Thus
\begin{align*}
    \Rgamma(\cC_{\rho,\gamma})\leq \inf_{\substack{\alpha\in(0,\frac{1}{2}), \beta\in(0,\alpha),\\\mu\in(\gamma\alpha,\infty),k\in\N}} \Rgamma(V(k, \alpha,\beta,\mu))\leq\tfrac{1}{2}.
\end{align*}
This is essentially due to the fact that the Riesz approximation rate decreases with the number of compositions $k$, whereas $k$ does not affect the rate for neural networks, but only the implicit constants of the asymptotics.  

\newpage

\appendix

\section{Excessive Orthogonality Example}\label{sec:lower_end}

\begin{lemma}
Let $\cH$ be a Hilbert space, let $(e_k)_{k\in\N}\subseteq\cH$ be an orthonormal system, $(s_k)_{k\in\N}$ be increasing, $(\delta_k)_{k\in\N}$ be decreasing, let $\rho\colon x\mapsto\max\{0,x\}$, and
\begin{align*}
     \C_{s,\delta}:=\{\sum_{k\in\N} \alpha_k e_k\colon |\{i\leq k\colon \alpha_i\neq 0\}|\leq s_k, |\alpha_k|\leq\delta_k \ \forall k\in\N\}.
\end{align*} 
If there exist $a,b,c_1,c_2,c_3\in(0,\infty)$ with $b>a$ such that, for every $k\in\N$, it holds that $\delta_k\in (c_1\log(k)^{-b},c_2 \log(k)^{-b})$ and $s_k\in [1,c_3\log(k)^{a})$, then
\begin{align}\label{eq:sgamma_minusinfty}
    \sgamma(\C_{s,\delta}) = -\infty.
\end{align}
Furthermore, if there exist $d,d'\in\N$, $\Omega\subset\R^d$, $q,p\in(0,1)$, $c\in(0,\infty)$, and $\pi\in\cP_+$, such that $\cH\subseteq \{f\colon\Omega\to\R^{d'}\}$ and that for every $\eps\in(0,\tfrac{1}{2})$, $k\in\N$, there is a neural network $\psi_{\eps,k}\in\cN_{d,d'}$ which satisfies
$\|e_k-\cR_\rho(\psi_{\eps,k})\|_{\cH}\leq \eps$, $\cM(\psi_{\eps,k})\leq c\eps^{-q}\log(k)^{p}$, $\cL(\psi_{\eps,k})\leq\pi(\log(\cM(\psi_{\eps,k})))$, $\cB(\psi_{\eps,k})\leq\pi(\cM(\psi_{\eps,k}))$, then
\begin{align*}
    \gamma^{\cN}_{\rho}(\C_{s,\delta})\geq \tfrac{b-a}{a+p+qb}.
\end{align*}

\end{lemma}

\begin{proof}
We first prove \eqref{eq:sgamma_minusinfty}. Assume, towards a proof by contradiction, that there exists a dictionary $\D=(\phi_i)_{i\in\N}\subseteq \cH$, $C,r\in(0,\infty)$, and $\tau,M_0\in\N$ such that, for every $M\geq M_0$, we have
    \begin{align*}
        \eps_M(\C_{s,\delta}):=\sup_{f\in\C_{s,\delta}}\infp_{h\in \<\cD_M\>_M}\|f-h\|_{\cH}\leq C M^{-r},
    \end{align*}
where $\cD_M=(\phi_i)_{i\in[M^\tau]}$.
We note that for $k,M\in\N$,  it holds that
\begin{align*}
        \eps_M(\C_{s,\delta})
        &=\sup_{f \in \C_{s,\delta}} \inf_{\substack{I\subseteq [M^\tau], \\  |I| = M, \, |c_i|\leq\pi(M)}} \left\|f - \sum_{i \in I} c_i \phi_i\right\|_{\cH}\\
        &\geq \sup_{j\in[k]}  \inf_{\substack{I\subseteq [M^\tau], \\  |I| = M, \, c_i\in\R}}\|\delta_k e_j - \sum_{i\in I} c_i\phi_i\|_{\cH}\\
        &\geq \max_{j\in[k]} \min_{h\in S_M} \|\delta_k e_j - h\|_{\cH},
\end{align*}
where $S_M$ is the span of $\phi_1,\dots,\lfloor M^\tau\rfloor$. Applying Lemma \ref{lem:aux1} we obtain, for all $k\in\N$, $M\geq M_0$ that
\begin{align*}
 \delta_k (1 - \sqrt{\tfrac{\lfloor M^\tau\rfloor}{k}}) \leq CM^{-r}.
\end{align*}
Rearranging the inequality and plugging in $\delta_k\geq c_1\log(k)^{-b}$ yields
\begin{align*}
     \sqrt{\tfrac{M^\tau}{k}} \geq 1 - \tfrac{C}{c_1}\log(k)^b M^{-r}.
\end{align*}
In particular this must hold for every $M\geq M_0$ and 
\begin{align*}
k_M:=\lfloor 2^{((\frac{c_1M^r}{2C})^{\frac{1}{b}})} \rfloor,
\end{align*}
in which case we get 
\begin{align*}
    \tfrac{C}{c_1}\log(k_M)^b M^{-\gamma}\leq \tfrac{1}{2},
\end{align*}
and consequently
\begin{align*}
    M^\tau\geq \tfrac{k_M}{4} \geq 2^{((\frac{c_1M^r}{2C})^{\frac{1}{b}})-3}.
\end{align*}
Since $M^\tau \in o(2^{((\frac{c_1M^r}{2C})^{\frac{1}{b}})-3})$ we have established a contradiction and thus proved \eqref{eq:sgamma_minusinfty}.

In order to prove \eqref{eq:Ngamma_pos}, we first note that $\cC_{s',\delta}\subseteq\cC_{s,\delta}$ if $s'_k\leq s_k$ for every $k\in\N$, and we can therefore w.l.o.g.\@ assume that $s_k=c_3\log(k)^{a}$.
We proceed by estimating the truncation error for 
\begin{align*}
    f=\sum_{k=1}^\infty \alpha_k e_k\in\C_{s,\delta}.
\end{align*}
Observe that, for $n\geq 3$, we have
\begin{align*}
    T_n(f):=\|\sum_{k=n+1}^\infty \alpha_k e_k\| \leq \lceil s_n\rceil \delta_n +\sum_{ s = \lceil s_n\rceil+1}^\infty \max_{j\colon s_j> s-1} \delta_j,
\end{align*}
where for the first $\lceil s_n\rceil$ many non-zero coefficients in the above sum we only use that they cannot be larger than $\delta_n$, whereas for any further coefficient we use that the $s$-th non-zero coefficient must have an index $j$ such that $s_j>s-1$.\\
Observe that $(\delta_k)_{k\in\N}$ is a decreasing sequence, and thus it holds for $s\in\N$ that
\begin{align*}
    \max_{j\colon s_j >s} \delta_j = \max_{j\colon j > 2^{((c_3^{-1}s)^{\frac{1}{a}})}}\delta_j\leq c_2\log(2^{((c_3^{-1}s)^{\frac{1}{a}})})^{-b}=c_2 c_3^{\frac{1}{a}}s^{-\frac{b}{a}}.
\end{align*}
We further note that for $t\in\N$ and $z>1$ it holds that
\begin{align*}
    \sum_{s=t}^\infty s^{-z} \leq\sum_{s=t}^\infty \int_{s-1}^s x^{-z}\mathrm{d}x =\int_{t-1}^\infty  x^{-z}\mathrm{d}x
    \leq \frac{1}{z-1}(t-1)^{-z+1}.
\end{align*}
Combining the above we get
\begin{align*}
    T_n(f)
    &\leq \lceil s_n\rceil \delta_n +\sum_{ s = \lceil s_n\rceil+1}^\infty \max_{j\colon s_j> s-1} \delta_j\\
    &\leq c_2c_3\lceil\log(n)^a\rceil \log(n)^{-b} + \sum_{ s = \lceil s_n\rceil}^\infty \max_{j\colon s_j> s} \delta_j\\
    &\leq 2c_2c_3\log(n)^{a-b} + \sum_{ s = \lceil s_n\rceil}^\infty c_2 c_3^{\frac{1}{a}}s^{-\frac{b}{a}}\\
    &\leq 2c_2c_3\log(n)^{a-b} + c_2 c_3^{\frac{1}{a}}\tfrac{1}{\frac{b}{a}-1}(\lceil s_n \rceil -1)^{-\frac{b}{a}+1}\\
    &\leq 2c_2c_3\log(n)^{a-b} + c_2 c_3^{\frac{1}{a}}\tfrac{a}{b-a}(\tfrac{c_3}{2})^{-\frac{b}{a}+1}\log(n)^{a-b}\\
    &=C_{a,b} \log(n)^{a-b},
\end{align*}
where $C_{a,b}:= 2c_2c_3+ c_2 c_3^{\frac{1}{a}}\tfrac{a}{b-a}(\tfrac{c_3}{2})^{-\frac{b}{a}+1}$.
We can now employ \Cref{NN_lincomb} to obtain, for every $f\in\C_{s,\delta}$ and $\eps\in(0,1/2)$, a network $\Psi_{f,\eps}$ satisfying
\begin{align*}
    \cR_\rho(\Psi_{f,\eps})=\sum_{k\in[n], \alpha_k \neq 0} \alpha_k \cR_\rho(\psi_{\eta,k})
\end{align*}
with
\begin{align*}
    n=n_\eps:= \left\lceil 2^{((\frac{\eps}{2C_{a,b}})^{\frac{1}{a-b}})}\right\rceil,\quad \eta=\eta_\eps:=\tfrac{\eps}{2 s_{n}}.
\end{align*}
Consequently we obtain
\begin{align*}
    \|f- \cR_\rho(\Psi_{f,\eps})\|&\leq T_n + \|\sum_{k\in[n], \alpha_k \neq 0} \alpha_k(e_k - \cR_\rho(\psi_{\eta,k}))\|
\leq C_{a,b}\log(n)^{a-b} + s_n \eta
\leq \eps.
\end{align*}
Next we estimate their connectivity as
\begin{align*}
    \M(\Psi_{f,\eps})
    &\leq2 \sum_{k\in[n], \alpha_k \neq 0} (\M(\psi_{\eta,k}) + \cL(\psi_{\eta,k}))\\
    &\leq 2s_n(c\eta^{-q}\log(n)^p + \pi(\log(c\eta^{-q}\log(n)^p)) \\
    &\in\cO_{\eps\to0}(\eps^{-\frac{a+p+qb}{b-a}}).
\end{align*}
Moreover, we have $\cL(\Psi_{f,\eps})\leq\max_{k\in[n]}\cL(\psi_{\eta,k})$ and $\cB(\Psi_{f,\eps})\leq \max_{k\in[n]}(|\alpha_k|\max\{1,\cB(\psi_{\eta,k})\})$, and since $\max_{k\in[n]}\cM(\psi_{\eta,k})\leq \cM(\Psi_{f,\eps})$, there clearly exists $\pi'\in\cP_+$ such that $\cL(\Psi_{f,\eps})\leq\pi'(\log(\cM(\Psi_{f,\eps})))$ and $\cB(\Psi_{f,\eps})\leq\pi'(\cM(\Psi_{f,\eps}))$. Therefore, we have
\begin{align}\label{eq:Ngamma_pos}
    \gamma^{\cN}_{\rho}(\C_{s,\delta})\geq \tfrac{b-a}{a+p+qb}.
\end{align}
\end{proof}

\section{Technical Lemmas and Proofs}\label{sec:proofs}
\begin{lemma}[Composition of ReLU neural networks]\label{lem:ReLUcomp}
    Let $\rho\colon x\mapsto\max\{0,x\}$, $d_1,d_2,d_3\in\N$, $\Phi_1\in \cN_{d_1,d_2}$, and $\Phi_2\in \cN_{d_2,d_3}$. Then there exists a network $\Psi\in\cN_{d_1,d_3}$ with $\cL(\Psi)=\cL(\Phi_1)+\cL(\Phi_2)$, $\cM(\Psi)\leq 2\cM(\Phi_1)+2\cM(\Phi_2)$, $\cB(\Psi)=\max\{\cB(\Phi_1),\cB(\Phi_2)\}$, and
    \begin{align*}
        \cR_\rho(\Psi) = \cR_\rho(\Phi_2)\circ\cR_\rho(\Phi_1).
    \end{align*}
\end{lemma}
\begin{proof}
    See Lemma II.3 in \cite{DeepTheory}.
\end{proof}

\begin{lemma}\label{NN_lincomb}
    Let $n,d,d',L\in\N$, and, for $i\in[n]$, let $a_i\in\R$ and $\Phi_i\in\cN_{d,d'}$ with $\cL(\Phi_i)\leq L$. Then there exists a network $\Psi\in\cN_{d,d'}$ with $\cL(\Psi)=L$, $\cM(\Psi)\leq 2\sum_{i\in[n]}(\cM(\Phi_i)+\cL(\Phi_i))$, $\cB(\Psi) \leq \max_{i\in[n]}(|a_i|\max\{1,\cB(\Phi_i)\})$, and 
    \begin{align*}
        \cR_\rho(\Psi)=\sum_{i\in[n]}a_i\cR_\rho(\Phi_i).
    \end{align*}
\end{lemma}
\begin{proof}
    This follows directly by combining Lemmas II.4 and A.8 in \cite{DeepTheory}.
\end{proof}

\begin{lemma}\label{lem:MbGoodEnough_Riesz}
       Let $T\subseteq\N$, $c\in(0,\infty)$, $M\in\N$, and  
       $\cD=(\phi_i)_{i\in T}\subset\cH$ be $c$-Riesz, and $x\in \langle \cD\rangle_M$ with $\|x\|=1$. 
       Then there exists $j\in T$ such that
       \begin{align}\label{eq:Riesz_dict_boudn}
           |\langle x,\phi_j\rangle|\geq \sqrt{\frac{c}{M}}.
       \end{align}
    \end{lemma} 
\begin{proof}
    As $x\in \langle \cD\rangle_M$, there exists $I\subseteq T$ with $|I|=M$ such that $x = \sum_{i\in I} \lambda_i\phi_i$. Due to the $c$-Riesz assumption and the fact that $\|\cdot\|_1 \leq \sqrt{M}\|\cdot\|_2$ on $\R^M$, we have
    \begin{align*}
        \sum_{i\in I}|\lambda_i|\leq \sqrt{M}\|\lambda\|_2 \leq \sqrt{\tfrac{M}{c}}. 
    \end{align*}
    Thus we have 
    \begin{align*}
        1 = |\<x,x\>|=|\sum_{i\in I}\lambda_i \<x,\phi_i\>|\leq (\max_{i\in I}|\<x,\phi_i\>|)\sum_{i\in I}|\lambda_i|\leq  \sqrt{\tfrac{M}{c}}\max_{i\in I}|\<x,\phi_i\>|.
    \end{align*}
    This ensures that \eqref{eq:Riesz_dict_boudn} is satisfied for $j\in \argmax_{i\in I}|\<x,\phi_i\>|.$

\end{proof}

\begin{lemma}\label{lem:aux1}
Let $\delta\in\R$, $d,k\in\N$ with $d<k$, let $\cH$ be a Hilbert space, let $(f_j)_{j=1}^k\subseteq \cH$ be an orthogonal family with $\|f_j\|_{\cH} = \delta$ for $j\in\{1,\dots,k\}$, and let $S\subseteq \cH$ be a $d$-dimensional linear subspace. Then
\begin{align*}
    \max_{j\in[k]}\min_{s\in S}\|f_j-s\|_{\cH}\geq \delta(1-\sqrt{\tfrac{d}{k}}).
\end{align*}

\end{lemma}

\begin{proof}
We first consider the case where $S\subseteq \<f_1,\dots,f_k\>$. Let $e_1,\dots,e_d$ be an orthonormal basis of $S$ and note that there are $a_{i,j}\in \R$, $i\in[d]$, $j\in[k]$, such that, for $i\in[d]$, we have
\begin{align}\label{phi_rep}
    e_i=\sum_{j=1}^k a_{i,j} f_j.
\end{align}
Next, we define, for every $j\in[k]$,
\begin{align*}
s_j:=\argmin_{s\in S}\|f_j-s\|_{\cH}
\end{align*}
and note that it is given by the orthogonal projection of $f_j$ into $S$, i.e.\@ we have
\begin{align*}
    s_j
    =\sum_{i=1}^d\ip{f_j,e_i}e_i
    =\sum_{i=1}^d\<f_j,\sum_{\ell=1}^k a_{i,\ell} f_\ell\>e_i
    =\sum_{i=1}^d\left(\sum_{\ell=1}^k a_{i,\ell}\ip{f_j,f_\ell} \right)e_i
    =\sum_{i=1}^d \|f_j\|_{\cH}^2 a_{i,j}e_i.
\end{align*}
Consequently we have, for every $j\in[k]$, that
\begin{align*}
    \|f_j-s_j\|_{\cH} 
    \geq \|f_j\|_{\cH} - \|f_j\|_{\cH}^2 \|\sum_{i=1}^d a_{i,j}e_i\|_{\cH} 
    =\delta -\delta^2(\sum_{i=1}^d |a_{i,j}|^2)^{\frac{1}{2}}
\end{align*}
We proceed by claiming that there must exist a $j\in[k]$ such that $\sum_{i=1}^d|a_{i,j}|^2\leq \tfrac{d}{\delta^2 k}$. Assume, towards a proof by contradiction, that $\sum_{i=1}^d|a_{i,j}|^2 > \tfrac{d}{\delta^2 k}$ for every $j\in[k]$. 
Note that \eqref{phi_rep} ensures $\sum_{j=1}^k |a_{i,j}|^2=\tfrac{1}{\delta^2}$. Thus we get
\begin{align*}
    \tfrac{d}{\delta^2}
    =k\tfrac{d}{\delta^2 k} 
    < \sum_{j=1}^k\sum_{i=1}^d|a_{i,j}|^2
    =\sum_{i=1}^d \sum_{j=1}^k|a_{i,j}|^2
    = d\sum_{j=1}^k|a_{i,j}|^2 
    =\tfrac{d}{\delta^2},
\end{align*}
which is a contradiction and thus proves our claim. As a consequence of this there exists $j\in[k]$ for which
\begin{align*}
        \|f_j-s_j\|_{\cH} \geq \delta -\delta^2 (\tfrac{d}{\delta^2 k})^{\tfrac{1}{2}} = \delta (1-\sqrt{\tfrac{d}{k}}),
\end{align*}
which concludes the proof in the case that $S\subseteq F:=\<f_1,\dots,f_k\>$.\\
In order to reduce the general case to the former, note that for an arbitrary $d$-dimensional linear subspace $S\subseteq \cH$ and every $s\in S$, it holds that 
\begin{align*}
    \|f_j-P_F(s)\|_{\cH} =\| P_F (f_j - s)\|_{\cH}\leq\|P_F\|_{\text{op}}\|f_j-s\|\leq\|f_j-s\| ,
\end{align*}
where $P_F$ is the orthogonal projection into $F$ and $\|\cdot\|_{\text{op}}$ the usual operator norm, which is bounded by $1$ for orthogonal projections. 
Thus we have that 
\begin{align*}
     \max_{j\in[k]}\min_{s\in S}\|f_j-s\|_{\cH} \geq  \max_{j\in[k]}\min_{s\in P_F(S)}\|f_j-s\|_{\cH},
\end{align*}
which completes the proof. 
\end{proof}

\begin{proof} [Proof of \Cref{lem:main_technical}]
         Let $y$ be the orthogonal projection of $x$ into $\langle z_1,\dots,z_T \rangle$ and note that $\|y\|_2\leq \|x\|_2=1$ and $|\<y,z_j\>|=|\<x,z_j\>|$ for all $j\in[T]$. Since $y \in \langle z_1,\dots,z_T \rangle$ there exists $\lambda=(\lambda_1,\dots,\lambda_T)\in\R^T$ such that
         \begin{align*}
             y=\sum_{j=1}^T \lambda_j z_j.
         \end{align*}
         W.l.o.g.\@ we assume that $|\lambda_1|\geq|\lambda_2|\geq\dots\geq|\lambda_T|$ and note that 
         \begin{align*}
             \sum_{\substack{i,j\in[T]\\i\neq j}}|\lambda_i \lambda_j|= 2\sum_{i=1}^{T-1}\sum_{j=i+1}^T |\lambda_i \lambda_j|\leq 2\sum_{i=1}^{T-1} (T-i)|\lambda_i|^2\leq 2T\|\lambda\|^2_2.
         \end{align*}
         Consequently we have
         \begin{align*}
             1
             &\geq \|y\|_2^2
             =\<\sum_{j=1}^T \lambda_j z_j,\sum_{j=1}^T \lambda_j z_j\>\\
             &=\sum_{j=1}^T (\lambda_j)^2 \<z_j,z_j\> + \sum_{\substack{i,j\in[T]\\i\neq j}}\lambda_i \lambda_j \<z_i,z_j\>\\
             &\geq \sum_{j=1}^T (\lambda_j)^2 - \sum_{\substack{i,j\in[T]\\i\neq j}}|\lambda_i \lambda_j \<z_i,z_j\>|\\
             &\geq \|\lambda\|^2_2 - 2T^{1-\theta}\|\lambda\|^2_2,             
         \end{align*}
         which implies
         \begin{align}\label{eq:LambdaEst}
             \|\lambda\|^2_2\leq \frac{1}{1-2T^{1-\theta}}.
         \end{align}
         Next we note that $|\lambda_1|\geq|\lambda_2|\geq\dots\geq|\lambda_T|$ implies $|\lambda_T|\leq T^{-\frac{1}{2}}\|\lambda\|_2$. Denoting by $P$ the orthogonal projection onto $\<z_T\>$, we obtain      
         \begin{align*}
             |\<y,z_T\>|&=|\sum_{j=1}^T \lambda_j \<z_j,z_T\>| \leq |\lambda_T| + T^{-\theta}\sum_{j=1}^{T-1} |\lambda_j|
             \leq |\lambda_T| + T^{-\theta} \|\lambda\|_1
             \leq T^{-\frac{1}{2}}\|\lambda\|_2 + T^{\frac{1}{2}-\theta} \|\lambda\|_2,
         \end{align*}
         where we used that $\|\lambda\|_1\leq T^{\frac{1}{2}}\|\lambda\|_2$, for $\lambda\in\R^{T}$, and that $\|P z_j\|_2=|\<z_T,z_j\>|\leq T^{-\theta}$, for $j\in[T-1]$, by assumption.
         Combining this with \eqref{eq:LambdaEst} yields
        \begin{align*}
            |\<y,z_T\>|\leq T^{-\frac{1}{2}}\frac{1+T^{1
            -\theta}}{(1-2T^{1-\theta})^{\frac{1}{2}}}\leq 4 T^{-\frac{1}{2}},
        \end{align*}
        since $1+T^{1-\theta}\leq 2$ and $(1-2T^{1-\theta})^{\frac{1}{2}}\geq \tfrac{1}{2}$ due to the assumption of $T\geq 4^{\frac{1}{\theta-1}}$.
    \end{proof}

\bibliography{bib}
\bibliographystyle{ieeetr}
\end{document}